\documentclass[final,12pt]{amsart}
\usepackage{geometry}
\usepackage{graphicx}
\usepackage[colorlinks=true, citecolor=blue, linkcolor=blue]{hyperref}
\newgeometry{asymmetric, centering}

\usepackage{ifdraft}
\ifoptionfinal{
\usepackage[disable]{todonotes}
}{

\usepackage[notref, notcite]{showkeys}
\usepackage[bordercolor=white, color=white]{todonotes}
}

\makeatletter\providecommand\@dotsep{5}\def\listtodoname{List of Todos}\def\listoftodos{\hypersetup{linkcolor=black}\@starttoc{tdo}\listtodoname\hypersetup{linkcolor=blue}}\makeatother

\newtheorem{lemma}{Lemma}
\newtheorem{proposition}{Proposition}
\newtheorem{theorem}{Theorem}

\def\R{\mathbb R}

\def\N{\mathbb N}

\def\p{\partial}

\newcommand{\pair}[1]{\left\langle #1 \right\rangle}
\newcommand{\norm}[1]{\left\|#1 \right\|}

\newcommand{\tnorm}[1]{\vert\hspace{-0.3mm}\Vert#1\Vert\hspace{-0.3mm}\vert}

\title[Data assimilation method for the heat equation]{Fully discrete finite element data assimilation method for the heat equation}
\author[E. Burman]{Erik Burman}
\address{Department of Mathematics, University College London, 
Gower Street, London UK, WC1E 6BT.}
\email{e.burman@ucl.ac.uk}
\author[J. Ish-Horowicz]{Jonathan Ish-Horowicz} 
\address{MRC Biostatistics Unit,
University of Cambridge,
Cambridge Biomedical Campus,
Cambridge UK, CB2 0SR}
\email{jonathan.horowicz@mrc-bsu.cam.ac.uk}
\author[L. Oksanen]{Lauri Oksanen}
\address{Department of Mathematics, University College London, 
Gower Street, London UK, WC1E 6BT.}
\email{l.oksanen@ucl.ac.uk}

\date{\today}

\begin{document}

\begin{abstract}
We consider a finite element discretization for the reconstruction of the final state of the heat equation, 
when the initial data is unknown, but additional data is given in a sub domain in the space time. For the discretization in space we consider standard
continuous affine finite element approximation, and the time derivative
is discretized using a backward differentiation. We regularize the
discrete system by adding a penalty of the $H^1$-semi-norm of the
initial data, scaled with the mesh-parameter.
The analysis of the method uses techniques developed in E. Burman and L. Oksanen {\em Data assimilation for the heat equation using stabilized finite element methods
}, arXiv, 2016, combining discrete stability of the numerical method
with sharp Carleman estimates for the physical problem, to derive optimal error estimates for the
approximate solution. For the natural space time energy norm,
away from $t=0$, the convergence is the same as for the classical problem with known initial data,
but contrary to the classical case,
we do not obtain faster convergence for the $L^2$-norm at the final time. 
\end{abstract}
\maketitle


\section{Introduction}

Time discretization of parabolic problems, discretized in space using
finite element methods, is a well studied topic, see for example the
monograph by Thom\'ee \cite{Thom97}. The analysis for all such methods
relies on the satisfaction of the hypothesis of the Lions theorem
\cite{Lions61}, stating the existence, uniqueness
and stability properties of the problem. 

The classical problem can be cast in the abstract form, find $u\in V$
such that
\begin{align}\label{}
\label{eq:abstract_parabolic}
&(\partial_t u, v)_H + a(u,v) = \left<f,v\right>_{V',V},
\\\label{eq:inital_data}
&u(0) = u_0 \in H,
\end{align}
where $V,\,H$ are some Hilbert spaces,  with $V$ dense in $H$ and imbedded with continuous identity, 
$\left<\cdot,\cdot\right>_{V',V}$ denotes the
duality pairing between $V$ and its dual, and $a(u,v):V\times V \mapsto
\mathbb{R}$ a symmetric
bilinear form representing the weak form of a second order
differential operator.
A key ingredient of the theory is that the
spatial operator satisfies the the G\aa rding's inequality, there are $\alpha>0$ and $\beta \ge 0$ such that for all $v \in
V$ there holds
\begin{equation}\label{eq:gaarding}
a(v,v) \ge \alpha \|v\|_V^2 - \beta \|v\|_H^2.
\end{equation}

In many situations for instance in environmental science and
meteorology the initial data is not available, instead some other data in the space time domain have been collected through measurements. 
This leads to a data assimilation problem,
that is, a problem to incorporate the observations of the physical system into the state of a computational model of the system. 
Computations can not be based on the classical theory,
since the equation (\ref{eq:inital_data}) can not be enforced 
when $u_0$ is not known.

It is then an interesting problem in computational
mathematics what quantities can be approximated and what is the effect of measurement errors on such an
approximation. 
The approximation methods need to take in 
the account the fact that these data assimilation problems are ill-posed in the sense that a necessary condition for them to be solvable 
is that the observations indeed come from the system. In other words, it must be assumed apriori that the solution exists,
and the mathematical theory concerns only uniqueness and stability. 

In \cite{BO16}, we studied finite element methods for 
two data assimilation problems with unknown $u_0$. 
The two problems differ in the sense that 
the lateral boundary data for $u$ is either known or unknown.
In the first case (\ref{eq:gaarding})
holds, whereas unknown lateral boundary data leads to a failure of
\eqref{eq:gaarding}.
This again gives rise to very different stability properties.
When the lateral boundary is known, the data assimilation problem is Lipschitz stable in suitable spaces, 
but the optimal stability is of conditional H\"older type when 
no information is given on the lateral boundary. 
Here we restrict our attention to the case with known later boundary data, and 
extend the corresponding results of \cite{BO16} to a fully discrete method. 
In \cite{BO16} discretization only in space was
considered.

The fully discrete analysis does not reduce straightforwardly to the semi-discrete case, as demonstrated by the fact that,
in order to achieve the optimal convergence rate with respect to the size of the time step, an additional regularization term is needed, see Theorem \ref{th_main} below.
There we consider two different asymptotic rates, $\tau = \mathcal O(h)$
and $\tau = \mathcal O(h^2)$, between the size of the finite element mesh $h$ and the time step $\tau$, 
and the analysis under the less restrictive rate $\tau = \mathcal O(h)$ is valid only when additional regularization is present (the case $\gamma_1>0$ in the theorem). 
In Section \ref{sec_comp}, we give a computational example showing
that the additional regularization is necessary.

To keep the exposition simple, we assume that the physical system is modelled by the heat equation
\begin{equation}\label{heat}
\partial_t u -\Delta u = f \quad \mbox{ in }  (0,T) \times \Omega,
\end{equation}
with $u = 0$ on the boundary $\partial \Omega$.
Here $\Omega \subset \mathbb{R}^d$ is a connected polyhedral
domain. 
Of course, in the absence of additional information, 
the equation (\ref{heat}) does not have a unique solution.
We assume that measurements of $u$, denoted by $q$,  are available in the
space time domain $(0,T) \times \omega$, where $\omega$ is a non-empty, open subset of
$\Omega$. 
We want to solve (\ref{heat}) under the additional constraint that 
\begin{equation}\label{Mdata}
u = q \quad \mbox{ in } (0,T) \times \omega.
\end{equation}
It is known that if there exists a solution $u$ to the equations 
(\ref{heat}) and (\ref{Mdata}), then the solution is unique.

A convenient way of solving the problem (\ref{heat})-(\ref{Mdata}) is through optimization. Casting the
problem in a form where the distance to the measured data in some norm
is minimised under the constraint of the heat equation, lead to a 4DVAR type
method. Such methods are important in data assimilation for
meteorology and environmental science and we refer to \cite{QJ:QJ340,QJ:QJ49712051912,dimet1986variational} for some 
results in the applied sciences. Although these methods are widely
used and popular tools, there appears to be no rigorous numerical
analysis assessing discretisation errors for them. One objective of the present publication is to
start filling this gap.

We will now discuss the previous mathematical literature on the problem (\ref{heat})-(\ref{Mdata}). We focus on techniques that work in dimensions $1+d$ with $d > 1$, and refer to the paper \cite{Wang2010} and references therein for 
the $1+1$-dimensional case. 
Our finite element method builds on the 
stability estimate \cite{Emanuilov1995}, and in a wider context,
the literature on continuum stability estimates for parabolic data assimilation (or unique continuation) problems is reviewed in 
\cite{Isakov2006, Yamamoto2009}. 

Computational methods for the problem (\ref{heat})-(\ref{Mdata})
go back to \cite{Lattes1967} where the quasi-reversibility method was introduced. 
Variations of this method for parabolic problems 
were developed in \cite{Klibanov2006, Klibanov1990, Tadi2002}
and in \cite{Becache2015}, and we refer to \cite{Klibanov2013} for a review of the quasi-reversibility method outside the parabolic context. 
Although for example the papers \cite{Klibanov2006, Becache2015}
consider convergence with respect to a Tikhonov type regularization parameter, none of the above papers prove convergence rates with respect to the refinement of a discretization. 
Proving such a convergence rate is the main novelty 
of the present paper. Moreover, compared to the previous literature, an attractive feature of our
method is that no auxiliary
Tikhonov type regularization parameters need to be introduced, the only asymptotic parameters are the size of the finite element mesh in space and the size of the time step.

Both the quasi-reversibility method and our method are based on 
Carleman estimates for the continuous problem. 
An alternative approach is to derive Carleman estimates directly on the discrete level, see for example \cite{BHR11} where such an approach was used for the closely related null controllability problem for the heat equation. 

The approach in the present paper has grown out of the study of stabilized finite element methods for unique continuation problems for elliptic equations \cite{Bu13,Bu14, BHL16}.
Another line of research that appears to be converging to a similar optimization based approach originates from the numerical analysis of the exact controllability of the wave equation \cite{Castro2014,Cindea2013,Cindea2015}. The approach has been applied to stable unique continuation problems for the wave equation \cite{Cindea2015a,Cindea2016} and 
to the null controllability problem for the heat equation \cite{Muench2016}.
Drawing from this line of research, a numerical analysis of the data assimilation problem for the heat equation is in preparation \cite{Muench2016b}, based on the continuous mixed formulation \cite{Muench2016a}.

\section{Discrete optimization problem}

Following \cite{BO16}, we first discretize \eqref{heat} in space only. Let $\mathcal{T}_h$ be
a conforming triangulation of the polyhedral domain $\Omega$. Let $h_K = \mbox{diam}(K)$ be
the local mesh parameter and $h = \max_{K\in \mathcal{T}_h} h_K$ the
mesh size. We assume that the family of triangulations
$\{\mathcal{T}_h\}_h$ is quasi uniform in the sense that there exists
a constant $c_1$ such that for all $K \in \mathcal{T}_h$ 
it holds that $
h_K \leq h \leq c_1 h_K$. Let $V_h$ be the standard
space of piecewise affine continuous finite elements 
satisfying the zero boundary condition,
\[
V_h = \{v \in H^1_0(\Omega);\ v \vert_{K} \in \mathbb{P}_1(K), \;
\forall K \in \mathcal{T}_h \}.
\]
We may then write a semi-discrete finite element formulation of \eqref{heat} as follows, find
$u \in C^1(0,T; V_h)$ such that
\begin{equation}\label{FEM}
(\partial_t u , v)+ a(u,v) = (f,v),
\quad v \in V_h,
\end{equation}
where 
\[
(u,v) = \int_\Omega u v\, dx,
\quad 
a(u,v) =  \int_\Omega \nabla u\cdot \nabla v\, dx.
\]
The idea is then to minimize the distance to the data (\ref{Mdata}) under the
constraint of this dynamical system. 

In order to outline this idea,
let us consider the following preliminary Lagrangian functional,
\begin{equation}\label{Lagrange_space}
\mathcal{L}_0(u,z) := \frac12 \|u - q\|_{L^2((0,T) \times \omega)}^2 +
\int_0^T (\partial_t u , z)+ a(u,z) - (f,z) \, dt .
\end{equation}
Writing the Euler-Lagrange equations for $\mathcal{L}_0$
we arrive to the following problem, find $(u,z)$
such that
\begin{align*}
\left<\partial_{u} \mathcal{L}_0(u,z), v \right> &= \int_0^T
(\partial_t v , z)+ a(v,z) + (u -
  q,v)_\omega\, dt = 0,
\\
\left<\partial_{z} \mathcal{L}_0(u,z), w \right> &= \int_0^T
 (\partial_t u , w)+ a(u,w) - (f,w)
\, dt = 0
\end{align*}
for all $v,w$. 
Here $(\cdot, \cdot)_\omega$ is the inner product on $L^2(\omega)$.
Clearly, if $z=0$ and $u$ solves
\eqref{FEM} with $u\vert_{(0,T) \times \omega} = q$, then these equations are
satisfied, and hence they are consistent with the data assimilation problem
that we set. This leads to a first possible approach: discretize this system in time and find the stationary
points of the discrete system. A numerical analysis however shows that this approach is
unlikely to be successful as the term 
$(u - q,v)_\omega$ does not seem to give enough stability for the problem to converge, and indeed, our computational examples in Section \ref{sec_comp} verify this. 
Instead we add certain regularization terms in the fully discrete context that we will describe next.

Let $N \in \N$ and $\tau > 0$ satisfy $N \tau = T$,
and define $t_n = n \tau$.
Furthermore, define for $u = (u^n)_{n=0}^N \in V_h^{N+1}$,
$$
\p_\tau u^n = \frac{u^n- u^{n-1}} \tau, \quad n=1,\dots,N.
$$

Consider the Lagrangian $\mathcal L : V_h^{N+1} \times V_h^N \to \R$ defined by
\begin{align}\label{def_L}
\mathcal L(u,z) 
&= 
\frac 1 2 \gamma_M \tau \sum_{n=1}^N \norm{u^n - q^n}_\omega^2
+ \frac 1 2 \gamma_0 \norm{h \nabla u^0}^2
+ \frac 1 2 \gamma_1 \tau \sum_{n=1}^N \norm{\tau \nabla \p_\tau u^n}^2 
\\\notag&\quad
+ \tau \sum_{n=1}^N \left(
(\p_\tau u^n, z^n) + a(u^n, z^n) - (f^n, z^n)
\right),
\end{align}
where,
for fixed functions $f \in C(0,T; L^2(\Omega))$
and $q \in C(0,T; L^2(\omega))$,
$$
f^n = f(t_n), \quad q^n = q(t_n), \quad n=1,\dots,N.
$$
We make the standing assumption that the fixed constants
$\gamma_M, \gamma_0$ and $\gamma_1$ satisfy the following
\begin{align}\label{gamma_pos}
\gamma_M, \gamma_0 > 0
\quad \text{and} \quad 
\gamma_1 \ge 0. 
\end{align}

Defining the bilinear forms
\begin{align*}
A_1(u,w) &= \tau \sum_{n=1}^N \left( 
(\p_\tau u^n, w^n) + a(u^n, w^n)
\right),
\\
A_2((u,z),v) &=
\gamma_M \tau \sum_{n=1}^N (u^n, v^n)_\omega
+ \gamma_0 (h \nabla u^0, h \nabla v^0)
+ \gamma_1 \tau \sum_{n=1}^N (\tau \nabla \p_\tau u^n, \tau \nabla \p_\tau v^n)
\\&\quad
+ \tau \sum_{n=1}^N \left( 
(\p_\tau v^n, z^n) + a(v^n, z^n)
\right),
\end{align*}
the Euler-Lagrange equations for $\mathcal L$ are 
\begin{align}\label{normal}
A_1(u,w) = \tau \sum_{n=1}^N (f^n, w^n),
\quad 
A_2((u,z),v) = \gamma_M \tau \sum_{n=1}^N (q^n, v^n)_\omega.
\end{align}

We define the seminorms
\begin{align*}
\tnorm{u}_R^2
&= 
\gamma_M \tau \sum_{n=1}^N \norm{u^n}_\omega^2
+ \gamma_0 \norm{h \nabla u^0}^2
+ \gamma_1 \tau \sum_{n=1}^N \norm{\tau \nabla \p_\tau u^n}^2,
\\
\tnorm{u,z}_D^2
&=
\norm{z^1}^2 
+ \norm{z^N}^2 
+ \tau^2 \sum_{n=2}^N \norm{\p_\tau z^n}^2
+ \tau \sum_{n=1}^N \norm{\nabla z^n}^2
\\&\quad
+ \norm{h \nabla u^N}^2
+ h^2 \tau \sum_{n=1}^N \norm{\p_\tau u^n}^2
+ h^2 \sum_{n=1}^N \norm{\tau \nabla \p_\tau u^n}^2,
\\
\tnorm{v,w}_C^2
&= \tnorm{v}_R^2 + 
\tau \sum_{n=1}^N \norm{w^n}^2.
\end{align*}
Note that $\tnorm{\cdot}_D$ is, in fact, a norm on $V_h^{2N+1}$.
Also, if $\gamma_1 > 0$ then $\tnorm{\cdot}_R$ and $\tnorm{\cdot}_C$ are norms on $V_h^{N+1}$ and $V_h^{2N+1}$, respectively.
The system (\ref{normal}) has the following coercivity property.

\begin{proposition}
\label{prop_coer}
There is $C > 0$ such that for all 
$N \in \N$, $h > 0$ and 
$(u,z)$ in $V_h^{2N+1}$ there is $(v,w)$ in $V_h^{2N+1}$
satisfying
$$
\tnorm{u}_R^2 + \tnorm{u,z}_D^2
\le C\left( A_1(u, w)
+  A_2((u, z), v) \right) ,
\quad
\tnorm{v,w}_C \le C \tnorm{u}_R + C \tnorm{u,z}_D.
$$
\end{proposition}
\begin{proof}
We will show first that there is $\alpha > 0$
such that for all $(u,z) \in V_h^{2N+1}$
\def\z{\hat z}
\begin{align}\label{coer1}
\frac 1 2 \left( \tnorm{u}_R^2 + \alpha \tnorm{u,z}_D^2
\right) 
\le A_1(u, -z + \alpha h^2 \p_\tau u)
+ A_2((u, z), u + \alpha \z),
\end{align}
where 
$\p_\tau u = (\p_\tau u^n)_{n=1}^N \in V_h^{N}$
and 
$\z = (\z^n)_{n=0}^N \in V_h^{N+1}$
is defined by $\z^0 = 0$ and $\z^n = z^n$, $n=1,\dots,N$.
Observe that 
$$
\tnorm{u}_R^2 = A_1(u, -z)
+ A_2((u, z), u).
$$

The identity 
\begin{align}\label{disc_antid}
\tau \sum_{n=1}^N (\p_\tau u^n, u^n)
= \frac 1 2 \left( \norm{u^N}^2 - \norm{u^0}^2 \right)
+ \frac{\tau^2} 2 \sum_{n=1}^N \norm{\p_\tau u^n}^2
\end{align}
is the discrete analogue of 
$$
\int_0^T (\p_t u, u)\, dt = \frac 1 2 \left( \norm{u(T)}^2 - \norm{u(0)}^2 \right).$$
To derive (\ref{disc_antid}) we employ
the polarization identity 
\begin{align*}
\tau (\p_\tau u^n, u^n)
= \norm{u^n}^2 - (u^{n-1}, u^n)
= \norm{u^n}^2 - \frac 1 2 \left( \norm{u^{n}}^2 + \norm{u^{n-1}}^2 - \norm{u^n-u^{n-1}}^2  \right),
\end{align*}
and observe that there is a telescoping type cancellation.
Using the identity (\ref{disc_antid}) 
with the bilinear form $(\cdot, \cdot)$ replaced by $a(\cdot,\cdot)$,
we have
\begin{align*}
A_1(u, \p_\tau u) &= 
\tau \sum_{n=1}^N \left( 
\norm{\p_\tau u^n}^2 + a(u^n, \p_\tau u^n)
\right)
\\&= 
\tau \sum_{n=1}^N \norm{\p_\tau u^n}^2
+ \frac 1 2 \left( \norm{\nabla u^N}^2 - \norm{\nabla u^0}^2 \right)
+ \frac{\tau^2} 2 \sum_{n=1}^N \norm{\nabla \p_\tau u^n}^2.
\end{align*}
Observe that if $\alpha \le \gamma_0$
then $-\alpha h^2 \norm{\nabla u^0}^2/2$ is absorbed by $\tnorm{u}_R^2$.

We have 
\begin{align*}
A_2((u, z), \z)
&=
\gamma_M \tau \sum_{n=1}^N (u^n, z^n)_\omega
+ \gamma_1 \tau \sum_{n=1}^N (\tau \nabla \p_\tau u^n, \tau \nabla \p_\tau \z^n)
\\&\quad
+ \tau \sum_{n=1}^N \left( 
(\p_\tau \z^n, z^n) + \norm{\nabla z^n}^2
\right).
\end{align*}
The identity (\ref{disc_antid}) gives
\begin{align*}
\tau \sum_{n=1}^N (\p_\tau \z^n, z^n)
= 
\frac 1 2 \norm{z^N}^2 
+ \frac{\tau^2} 2 \sum_{n=1}^N \norm{\p_\tau \z^n}^2
= 
\frac 1 2 \norm{z^N}^2 
+ \frac 1 2 \norm{z^1}^2 
+ \frac{\tau^2} 2 \sum_{n=2}^N \norm{\p_\tau z^n}^2.
\end{align*}
Let us now consider the cross terms. The Poincar\'e inequality gives
\begin{align*}
(u^n, z^n)_\omega \le (4\delta)^{-1} \norm{u^n}_\omega^2 + C \delta \norm{\nabla z^n}^2,
\end{align*}
and the second term can be absorbed by $\norm{\nabla z^n}^2$
for small $\delta > 0$.
The first term is absorbed by $\tnorm{u}_R^2$ for small $\alpha > 0$. For the second cross term, 
\begin{align*}
\tau \sum_{n=1}^N (\tau \nabla \p_\tau u^n, \tau \nabla \p_\tau \z^n)
\le (2\delta)^{-1} \tau \sum_{n=1}^N \norm{\tau \nabla \p_\tau u^n}^2 + \delta \tau \sum_{n=1}^N \norm{\nabla z^n}^2
\end{align*}
and we see that these two terms are absorbed analogously with the above. 
This finishes the proof of (\ref{coer1}).

It remains to show that 
$$
\tnorm{v,w}_C \le C \tnorm{u}_R + C \tnorm{u,z}_D.
$$
when $v = u + \alpha \z$ and $w = -z + \alpha h^2 \p_\tau u$.
We have
\begin{align*}
\tnorm{\z}_R^2
&= 
\gamma_M \tau \sum_{n=1}^N \norm{z^n}_\omega^2
+ \gamma_1 \tau \sum_{n=1}^N \norm{\tau \nabla \p_\tau \z^n}^2
\le C \tau \sum_{n=1}^N \norm{\nabla z^n}^2
\le C \tnorm{0,z}_D^2,
\end{align*}
where the Poincar\'e inequality and the triangle inequality was used for the first and the second term, respectively.
Using the Poincar\'e inequality again, we have
$$
\tau \sum_{n=1}^N \norm{z^n}^2 \le C \tnorm{0,z}_D^2.
$$
The bounds for the terms containing $u$ are trivial.
\end{proof}

Denote by $N_h$ the dimension of $V_h$.
The equations (\ref{normal}) define a square linear system
of $(2N+1)N_h$ unknowns, and taking 
$f^n = 0$ and $q^n = 0$, $n=1,\dots,N$,
it follows from Proposition \ref{prop_coer} that 
$(u,z) = 0$ is the only solution of the corresponding homogeneous system. 
Thus (\ref{normal}) has a unique solution.

\section{A priori error estimates}

\begin{proposition}
\label{prop_tnorm}
Suppose that $\Omega$ is a convex polyhedral domain and that $u$ is in 
\begin{align}
\label{star_space}
H^1(0,T; H^1_0(\Omega))
\cap H^2(0,T;L^2(\Omega)).
\end{align}
Denote by $\norm{\cdot}_*$ the norm in (\ref{star_space}).
Let $(u_h,z_h) \in V_h^{2N+1}$
be the solution of (\ref{normal})
with $f = \p_t u - \Delta u$ and $q = u|_{(0,T) \times \omega}$, and 
suppose that $f \in C(0,T;L^2(\Omega))$.
Then
\begin{align*}
&\tnorm{\pi_h u - u_h}_R + \tnorm{\pi_h u - u_h, z_h}_D 
\le C (h+\tau) \norm{u}_*,
\end{align*}
where $\pi_h u$ is the orthogonal projection defined by 
\begin{align}\label{pi_ortho}
a(\pi_h u, w) = a(u,w), \quad w \in V_h.
\end{align}
\end{proposition}
\begin{proof}
We use the shorthand notation $\xi_h = \pi_h u - u_h$. 
By Proposition \ref{prop_coer} it is enough to show that 
$$
A_1(\xi_h, w) + A_2((\xi_h, z_h), v)
\le C (h + \tau) \tnorm{v,w}_C \norm{u}_*, \quad (v,w) \in V_h^{2N+1}.
$$

The point values $u^n = u(t_n)$ satisfy
$$
(\p_t u^n , \phi) + a(u^n , \phi) = (f^n, \phi), \quad n = 1,\dots,N,\ \phi \in H^1_0(\Omega).
$$
This implies the following consistency relation
\begin{align*}
A_1(u - u_h, w)
&=  \tau \sum_{n=1}^N \left( 
(\p_\tau u^n , w^n) + a(u^n , w^n)
\right)
- \tau \sum_{n=1}^N (f^n, w^n)
\\&=  \tau \sum_{n=1}^N (\p_\tau u^n - \p_t u^n, w^n).
\end{align*}
Using also the orthogonality (\ref{pi_ortho}), we get
\begin{align*}
A_1(\xi_h, w) 
&= 
A_1(\pi_h u - u, w) + A_1(u - u_h, w)
\\\notag&=
\tau \sum_{n=1}^N ((\pi_h - 1)\p_\tau u^n, w^n)
+ \tau \sum_{n=1}^N (\p_\tau u^n - \p_t u^n, w^n).
\end{align*}
The Cauchy-Schwarz inequality implies that
$A_1(\xi_h, w) \le 2 (I_1 + I_2)^{1/2} \tnorm{0,w}_C$ where
\begin{align*}
I_1 = \tau \sum_{n=1}^N \norm{(\pi_h - 1)\p_\tau u^n}^2,
\quad I_2 = \tau \sum_{n=1}^N \norm{\p_\tau u^n - \p_t u^n}^2.
\end{align*}
We estimate $I_1$ by using the approximation properties of $\pi_h$,
see e.g. \cite[Th. 3.16 and 3.18]{Ern2004},
\begin{align*}
I_1
&= 
\tau^{-1} \sum_{n=1}^N
\norm{\int_{t_{n-1}}^{t_n} (\pi_h - 1) \p_t u\, dt}^2
\le
\sum_{n=1}^N
\int_{t_{n-1}}^{t_n} \norm{(\pi_h - 1) \p_t u\, dt}^2
\\&\le 
 C h^{2} \int_0^T \norm{\nabla \p_t u}^2\, dt.
\end{align*}
For $I_2$
we use Taylor's theorem with the integral form of the remainder,
\begin{align*}
I_2
&= \tau^{-1} \sum_{n=1}^N
\norm{\int_{t_{n-1}}^{t_n} \frac {t_n - t} 2\, \p_t^2 u \, dt }^2
\le
\tau^{-1} \sum_{n=1}^N 
\int_{t_{n-1}}^{t_n} (t_n - t)^2\, dt 
\int_{t_{n-1}}^{t_n} \norm{\p_t^2 u}^2 dt
\\&\le
\tau^2 \int_0^T \norm{\p_t^2 u}^2 \, dt.
\end{align*}

Let us now turn to the second bilinear form. We have 
\begin{align*}
A_2((\xi_h,z_h),v) 
&=
\gamma_M \tau \sum_{n=0}^N (\pi_h u^n - u^n, v^n)_\omega
+ \gamma_0 (h \nabla \pi_h u^0, h \nabla v^0)
\\&\quad
+ \gamma_1 \tau \sum_{n=1}^N (\tau \nabla \p_\tau \pi_h u^n, \tau \nabla \p_\tau v^n).
\end{align*}
Thus
$A_2((\xi_h,z_h),v) \le C (I_3 + I_4 + I_5)^{1/2} \tnorm{v,0}_C$,
where
\begin{align}
I_3 &= \tau \sum_{n=0}^N \norm{\pi_h u^n - u^n}_\omega^2
\le h^{2} \tau \sum_{n=0}^N \norm{\nabla u^n}^2
\le C h^{2} \norm{\nabla u}_{H^1(0,T;L^2(\Omega))}^2, \nonumber
\\
I_4 &=
\norm{h \nabla \pi_h u^0}^2
\le C h^2 \norm{\nabla u}_{H^1(0,T;L^2(\Omega))}^2, \nonumber
\\
I_5 &= \tau \sum_{n=1}^N \norm{\nabla \pi_h \tau \p_\tau  u^n}^2
= \tau \sum_{n=1}^N \norm{\int_{t_{n-1}}^{t_n} \nabla \pi_h \p_t  u\, dt}^2
\le \tau^2 \int_0^T \norm{\nabla \p_t u}^2 dt. \label{eq:useful}
\end{align}
Here we used the trace inequality in time and the continuity of $\pi_h$.
\end{proof}

We recall the following variation of \cite{Emanuilov1995} that was proven 
in \cite{BO16}.

\begin{theorem}
\label{th_cont_stable}
Let $\Omega \subset \R^d$ be a convex polyhedron,
let $\omega \subset \Omega$ be open and non-empty, and let $0 < \delta < T$. 
Then there is $C > 0$ such that for all $u$ in the space
\begin{align}
\label{energy_space}
H^1(0,T; H^{-1}(\Omega)) \cap L^2(0,T; H_0^1(\Omega)),
\end{align}
it holds that 
\begin{align*}
&\norm{u}_\delta \le C (\norm{u}_{L^2((0, T) \times \omega)} + 
\norm{\p_t u - \Delta u}_{L^2(0, T; H^{-1}(\Omega))}),
\end{align*}
where $\norm{\cdot}_\delta$ is the norm in 
$C(\delta, T; L^2(\Omega)) \cap L^2(\delta, T; H^1(\Omega)) \cap H^1(\delta, T; H^{-1}(\Omega))$.
\end{theorem}

For $u_h = (u_h^n)_{n=0}^N \in V_h^{2N+1}$
we define the linear interpolation
\begin{align}\label{interpolation}
\tilde u_h(t) = \tau^{-1} \left(
(t-t_{n-1}) u^n_h + (t_{n}-t) u_h^{n-1}
\right), 
\quad t \in [t_{n-1},t_n],\ n = 1,\dots,N.
\end{align}
Observe that $\tilde u_h$ is in the space (\ref{energy_space})
and also in $C(0,T;H^1_0(\Omega))$.
We are now ready to prove our main result on the convergence of the stabilized finite element method. 

\begin{theorem}
\label{th_main}
Let $\omega \subset \Omega \subset \R^d$ and $\delta > 0$ 
be as in Theorem \ref{th_cont_stable}.
Let $u$, $f$ and $(u_h, z_h)$ be as in Proposition \ref{prop_tnorm}
and define $\tilde u_h$ by (\ref{interpolation}).
Suppose that $f \in H^1(0,T; L^2(\Omega))$.
Furthermore, in the case $\gamma_1 > 0$ suppose that 
$\tau = \mathcal O(h)$, and in the case $\gamma_1 = 0$
suppose that $\tau = \mathcal O(h^2)$.
Then 
$$
\norm{u - \tilde u_h}_\delta \le C h 
\left( 
\norm{u}_* + \norm{f}_{H^1(0,T; L^2(\Omega))}
\right).
$$
\end{theorem}
\begin{proof}
Let $e = u - \tilde u_h$, and define the linear form
$$
\pair{r, w}
= \int_0^T (\p_t e, w) + a(e,w)\, dt, 
\quad w \in L^2(0,T; H_0^1(\Omega)).
$$
By Theorem \ref{th_cont_stable} it is enough to show 
the following two inequalities 
\begin{align}\label{e_omega}
\norm{e}_{L^2((0, T) \times \omega)}
&\le C h \norm{u}_*,
\\\label{r}
\pair{r, w} 
&\le C h 
\left( 
\norm{u}_* + \norm{f}_{L^2((0,T) \times \Omega)}
\right)
\norm{w}_{L^2(0,T; H_0^1(\Omega))}.
\end{align}

Let us begin with (\ref{e_omega}).
We define the projection on the piecewise constant functions
$$
\pi_0 v(t) = v(t^n), 
\quad t \in (t_{n-1},t_n],\ n = 1,\dots,N.
$$
Observe that 
$$
\norm{\pi_0 v - v}_{L^2(0,T)} \le \tau 
\norm{\p_t v}_{L^2(0,T)}, 
\quad v \in H^1(0,T).
$$
We have 
\begin{align*}
\norm{e}_{L^2((0, T) \times \omega)}^2
\le C (h^2 + \tau^2)\norm{u}_{H^1(0,T;H^1(\Omega))}^2
+ \int_0^T \norm{\pi_0 \pi_h u - \tilde u_h}_\omega^2 dt,
\end{align*}
and
\begin{align*}
\int_0^T \norm{\pi_0 \pi_h u - \tilde u_h}_\omega^2 dt
&\le 
\int_{0}^{T} \norm{\pi_0 \pi_h u - \pi_0 \tilde u_h}_\omega^2 dt
+ \int_0^T \norm{\pi_0 \tilde u_h - \tilde u_h}_\omega^2 dt
\\&= 
\tau \sum_{n=1}^N \norm{\pi_h u^n - u_h^n}_\omega^2
+ \sum_{n=1}^N  \int_{t_{n-1}}^{t_n} \norm{\pi_0 \tilde u_h - \tilde u_h}_\omega^2 dt.
\end{align*}
Here the first term is bounded by $\tnorm{\pi_h u - u_h}_R^2$,
and we use the identity 
\begin{align}\label{interp_id}
\tilde u_h 
= u^n_h + (t-t_{n}) \p_\tau u_h^n
\end{align}
to estimate the second one as follows
\begin{align*}
\sum_{n=1}^N  \int_{t_{n-1}}^{t_n} \norm{\pi_0 \tilde u_h - \tilde u_h}^2 dt
&=
\sum_{n=1}^N  \int_{t_{n-1}}^{t_n} \norm{(t_{n}-t) \p_\tau u_h^n}^2 dt
\le 
\tau \sum_{n=1}^N \norm{\tau \p_\tau u_h^n}^2
\\&\le 
\tau \sum_{n=1}^N \norm{\tau \p_\tau (\pi_h u^n - u_h^n)}^2
+ \tau \sum_{n=1}^N \norm{\tau \p_\tau \pi_h u^n}^2.
\end{align*}
As $\tau = \mathcal O(h)$, the first term above is bounded by $\tnorm{\pi_h u - u_h,0}_D^2$,
and the second term is bounded by $\tau^2 \norm{u}^2_*$.
The inequality (\ref{e_omega}) follows from Proposition \ref{prop_tnorm}.

We turn to (\ref{r}), and define the piecewise constant function
defined by local time averages
\def\w{\overline w}
$$
\w(t) = \tau^{-1} \int_{t_{n-1}}^{t_n} w\, dt,
\quad t \in (t_{n-1},t_n],\ n = 1,\dots,N.
$$
We have
\begin{align*}
\int_0^T (\p_t u, w) + a(u,w)\, dt
= \int_0^T (f, w)\, dt = 
\int_0^T (f - \pi_0 f, w)\, dt + \tau \sum_{n=1}^N (f^n, \w),
\end{align*}
and using the identity (\ref{interp_id}) and the orthogonality (\ref{pi_ortho}),
\begin{align*}
&-\int_0^T (\p_t \tilde u_h, w) + a(\tilde u_h, w)\, dt
=
-\tau \sum_{n=1}^N (\p_\tau u_h^n, \w)
-\int_0^T a(\tilde u_h, \pi_h w)\, dt
\\&\quad=
-\tau \sum_{n=1}^N (\p_\tau u_h^n, \w)
-\tau \sum_{n=1}^N a(u_h^n, \pi_h \w)
-\sum_{n=1}^N \int_{t_{n-1}}^{t_n} (t-t_n)\, a(\p_\tau u_h^n, \pi_h w)\, dt.
\end{align*}
As $u_h$ satisfies (\ref{normal}), it holds that
\begin{align*}
\pair{r,w}
&=
\int_0^T (f - \pi_0 f, w)\, dt + 
\tau \sum_{n=1}^N (f^n, \w - \pi_h \w)
-\tau \sum_{n=1}^N (\p_\tau u_h^n, \w- \pi_h \w)
\\&\quad
-\sum_{n=1}^N \int_{t_{n-1}}^{t_n} (t-t_n)\, a(\p_\tau u_h^n,\pi_h w)\, dt.
\end{align*}
We have
\begin{align*}
\int_0^T (f - \pi_0 f, w)\, dt
&\le \tau \norm{f}_{H^1(0,T;L^2(\Omega))} \norm{w}_{L^2((0,T) \times \Omega)},
\\
\tau \sum_{n=1}^N (f^n, \w - \pi_h \w)
&\le 
C h \norm{f}_{H^1(0,T;L^2(\Omega))}
\norm{w}_{L^2(0,T; H^1(\Omega))}.
\end{align*}
Moreover,
\begin{align*}
\tau \sum_{n=1}^N (\p_\tau u_h^n, \w- \pi_h \w)
\le C h 
\norm{u}_{H^2(0,T;L^2(\Omega))}
\norm{w}_{L^2(0,T; H^1(\Omega))},
\end{align*}
where we used Proposition \ref{prop_tnorm}, after observing that
\begin{align*}
h^2 \tau \sum_{n=1}^N \norm{\p_\tau u_h^n}^2 
\le \tnorm{u_h - \pi_h u,0}_D^2 + h^2 \norm{u}_*^2.
\end{align*}
Finally,
\begin{align*}
\sum_{n=1}^N \int_{t_{n-1}}^{t_n} (t-t_n)\, a(\p_\tau u_h^n,\pi_h w)\, dt
\le \tau \left(
\tau \sum_{n=1}^N \norm{\nabla \p_\tau u_h^n}^2 
\right)^{\frac 1 2} \norm{w}_{L^2(0,T; H^1(\Omega))},
\end{align*}
and using the triangle inequality and \eqref{eq:useful},
\begin{align*}
\tau \sum_{n=1}^N \norm{\tau \nabla \p_\tau u_h^n}^2
\le 
\tau \sum_{n=1}^N \norm{\tau \nabla \p_\tau (u_h^n - \pi_h u^n)}^2
+ 
C \tau^2 \int_0^T \norm{\nabla \p_t u}^2 dt.
\end{align*}
Observe that 
\begin{align*}
\tau \sum_{n=1}^N \norm{\tau \nabla \p_\tau (u_h^n - \pi_h u^n)}^2
\le
C 
\begin{cases}
\tnorm{u_h - \pi_h u}^2_R, & \gamma_1 > 0,
\\[3mm]
\tnorm{u_h - \pi_h u, 0}^2_D, & \tau = \mathcal O(h^2).
\end{cases}
\end{align*}
The inequality (\ref{r}) follows from Proposition \ref{prop_tnorm}.
\end{proof}
If $\gamma_1=0$ and $\tau = \mathcal O(h)$ then Theorem \ref{th_main}
does not predict optimal convergence. Indeed, in this case the bound
in the last step becomes
\[
\tau \sum_{n=1}^N \norm{\tau \nabla \p_\tau (u_h^n - \pi_h u^n)}^2
\leq C h^{-1} \tnorm{u_h - \pi_h u, 0}^2_D.
\]
This then leads to a convergence of order
$O(h^{\frac12}+\tau^{\frac12})$ using Proposition \ref{prop_tnorm}.
\subsection{The case of perturbations in data}
Thanks to the Lipschitz stability of Theorem \ref{th_cont_stable} the
extension of the above analysis to the case where the data is
perturbed is straightforward. Indeed assume that instead of
$(q^n,f^n)_{n=1}^N$ in \eqref{def_L} we have at are disposal the
perturbed data $(\tilde q^n,\tilde f^n)_{n=1}^N$, 
\[
\tilde q^n = q^n + \delta q^n, \quad \tilde f^n = f^n + \delta f^n
\]
with $\delta q^n \in L^2(\omega)$ and $\delta f^n \in
H^{-1}(\Omega)$. Then a standard perturbation argument leads to
similar results as Proposition \ref{prop_tnorm} and Theorem
\ref{th_main},
but with an additional term of the form 
$$C \tau^{\frac12}
\left(\sum_{n=1}^N \left( \|\delta q^n\|^2_\omega + \|\delta
    f^n\|^2_{H^{-1}(\Omega)} \right) \right)^{\frac12}$$ in the right
hand side of the bounds of the error estimates. This is a similar
result as one would obtain for a well-posed problem.
\section{Computational examples}
\label{sec_comp}
The main objectives of the computational examples are twofold. 
\begin{enumerate}
\item First we verify that the predicted reduction in
  convergence order to $O(h^{\frac12}+\tau^{\frac12})$ for $\gamma_1=0$ and $\tau = \mathcal O(h)$ indeed
  takes place, even in a simple model case.
\item Then we confirm that the situation is rectified for $\gamma_1>0$.
\end{enumerate}

The Euler-Lagrange equations (\ref{normal}) form a non-singular, symmetric system of $(2 N + 1) N_h$ linear equations. 
We emphasize that the system is not positive definite. In principle, it can be solved using off-the-shelf methods, for example the MINRES method \cite{Paige1975}. 

We implemented this straightforward strategy in the case that 
$\gamma_1 = 0$, and verified that the convergence order in space is
that predicted by Theorem \ref{th_main}. For the convergence order in
time we verify that failure to meet the condition $\tau  =
\mathcal{O}(h^2)$ indeed leads to suboptimal convergence. We observe
$\mathcal{O}(\tau^{\frac12})$ convergence under refinement of $\tau$
in the regime where $\tau  =
\mathcal{O}(h)$.
In all our computational examples $\Omega$ is the unit interval $(0,1)$, 
$\omega = (a, 1-a)$, $a = 0.2$,
and we use a regular mesh on $\Omega$.
Moreover, the function $u$ is of the form
\begin{align}\label{u_comp}
u(t,x) = e^{-\pi^2 k^2 t} \sin (\pi k x), \quad k = 1,2.
\end{align}
Computations 
for $k = 2$ and $T = 0.02$
are summarized in Table \ref{tab_monolithic}.
We also verified that the computations diverge when no regularization is introduced, that is, when $\gamma_0 = 0$.
In these computations we used the MINRES implementation of SciPy with the default parameters \cite{Jones2001--}, and the initial guess was set to zero. The convergence is typically slow, requiring thousands of iterations. 

\begin{table}\centering
\begin{tabular}{ l | c  c  c   }
$h$ & 
0.02 & 0.01 & 0.005
\\\hline
error &  
0.224 & 0.119 & 0.043
\end{tabular}
\qquad
\begin{tabular}{ l |  c  c  c   }
$\tau$ & 
0.004 & 0.002 & 0.001
\\\hline
error &  
0.104 & 0.073 & 0.048
\end{tabular}
\medskip
\caption{Convergence with $\gamma_M = \gamma_0 = 1$ and $\gamma_1=0$
using the MINRES method. 
The error is $\norm{u(T) - u_h^N}_{L^2(\Omega)}$.
{\em Left.} Order $1$ convergence in $h$ with $N=16$.
{\em Right.} Order $1/2$ convergence in $\tau$ with $N_h=200$.
}
\label{tab_monolithic}
\end{table}

The remaining examples will exploit the structure of (\ref{normal}) to reduce the computational burden.

\subsection{The Euler-Lagrange equations as a system of two coupled heat equations}
An attractive feature of the regularization in (\ref{def_L})
is that it acts only on the primal variable $u$.
This leads to the one-way coupling in (\ref{normal}),
that is, the dual variable $z$ does not appear in the equation involving $A_1$. 
We present next a method solving (\ref{normal}) that is based on the one-way coupling. 

Note that the first equation in (\ref{normal}), that is,
\begin{align}\label{heat_u}
\tau \sum_{n=1}^N \left( 
(\p_\tau u^n, w^n) + a(u^n, w^n)
\right) = \tau \sum_{n=1}^N (f^n, w^n),
\end{align}
is simply a discretization of the heat equation (\ref{heat}).
Let us next interpret the second equation in (\ref{normal})
as a discretization of a heat equation for $z$.
Observe that, setting $z^{N+1} = 0$, we obtain
\begin{align*}
\tau \sum_{n=1}^N (\p_\tau v^n, z^n)
= - \tau \sum_{n=1}^N (v^n, \p_\tau z^{n+1}) - (v^0, z^1).
\end{align*}
Thus choosing $v^0 = 0$ in (\ref{normal}) for the moment, we see that $z$ satisfies
\begin{align}\label{heat_z}
&\tau \sum_{n=1}^N 
\left(
- (v^n, \p_\tau z^{n+1}) + a(v^n, z^n) 
\right)
\\\notag&\quad=
\gamma_M \tau \sum_{n=1}^N (q^n - u^n, v^n)_\omega
- \gamma_1 \tau \sum_{n=1}^N (\tau \nabla \p_\tau u^n, \tau \nabla \p_\tau v^n),
\end{align}
and this can be interpreted as a discretization of 
\begin{align*}
-\p_t z - \Delta z = \gamma_M (q- u) 1_\omega.
\end{align*}
Here $1_\omega$ is the indicator function of $\omega$, that is, $1_\omega(x) = 1$ if $x \in \omega$ and $1_\omega(x) = 0$ otherwise.
Note that, when rescaled by $\tau^{-2}$, the second term on the right-hand side of (\ref{heat_z}) is a discretization of 
$\int_0^T (\nabla \p_t u, \nabla \p_t v)\, dt$. 
Taking now $v^n = 0$, $n=1,\dots,N$, in (\ref{normal})
we get the additional constraint
\begin{align*}
\gamma_0 (h \nabla u^0, h \nabla v^0)
- \gamma_1 \tau (\tau \nabla \p_\tau u^1, \nabla v^0)
- (z^1, v^0) = 0.
\end{align*}

Define $U(\phi)$ to be the solution of (\ref{heat_u}) 
with $u^0 = \phi$, and $Z(\phi)$ the solution of (\ref{heat_z})
with $z^{N+1} = 0$ and $u = U(\phi)$.
Observe that these can be easily computed by using time stepping.
Furthermore, define the function
$$
\mathcal C(\phi, \psi) = 
\gamma_0 (h \nabla U^0(\phi), h \nabla \psi)
- \gamma_1 \tau (\tau \nabla \p_\tau U^1(\phi), \nabla \psi)
- (Z^1(\phi), \psi), \quad \psi \in V_h.
$$
Then $(u,z) = (U(\phi), Z(\phi))$ solves 
(\ref{normal}) if and only if 
\begin{align}\label{coupling}
\mathcal C(\phi, \psi) = 0, \quad \psi \in V_h.
\end{align}

We will use a gradient descent type method to solve (\ref{coupling}). 
Starting from an initial guess $\phi_0 \in V_h$, we define the iteration
\begin{align}\label{graddesc}
(\phi_{m+1}, \psi) = 
(\phi_m, \psi)
-\alpha \mathcal C(\phi_m, \psi), \quad \psi \in V_h,
\end{align}
where $\alpha > 0$ is a step size.
The system (\ref{graddesc}) is a discretization of 
the differential equation
\begin{align}\label{def_Phi}
\Phi(0) = \phi_0, \quad (\p_s \Phi(s), \psi) = 
-\mathcal C(\Phi(s), \psi), \quad \psi \in V_h,
\end{align}
and its use to solve (\ref{coupling}) is justified by the following lemma.

\begin{lemma}
Let $\phi_0 \in V_h$ and define a one parameter family $\Phi(s)$, $s \ge 0$, in $V_h$ by (\ref{def_Phi}).
Let $(u_h, z_h)$ be the solution of (\ref{normal}).
Then $\Phi(s)$ converges to $u_h^0$ as $s \to \infty$.
\end{lemma}
\begin{proof}
For each $s \ge 0$ 
it holds by definition that $u(s) = U(\Phi(s))$
and $z(s) = Z(\Phi(s))$ satisfy (\ref{heat_u})
and (\ref{heat_z}), respectively. Hence
\begin{align*}
\p_s \mathcal L(u, z) 
= 
(\p_u L, \p_s u) + (\p_z L, \p_s z) 
=
\mathcal C(\Phi, \p_s u^0)
= \mathcal C(\Phi, \p_s \Phi)
= -\norm{\p_s \Phi}^2.
\end{align*}
The equation (\ref{heat_u}) implies also that 
$$
\mathcal L(u,z) 
= 
\frac 1 2 \gamma_M \tau \sum_{n=1}^N \norm{u^n - q^n}_\omega^2
+ \frac 1 2 \gamma_0 \norm{h \nabla \Phi}^2
+ \frac 1 2 \gamma_1 \tau \sum_{n=1}^N \norm{\tau \nabla \p_\tau u^n}^2.
$$
As $\mathcal L$ is non-negative and decreasing 
along the family $(u(s), z(s))$, it follows that
$\p_s \mathcal L(u,z) \to 0$ as $s \to \infty$.
Hence also $\p_s \Phi \to 0$ as $s \to \infty$,
and the differential equation (\ref{def_Phi})
implies that the limit $\phi_\infty = \lim_{s \to \infty} \Phi(s)$ exists and satisfies (\ref{coupling}).
By the discussion preceding the proof, we have $\phi_\infty = u_h^0$.
\end{proof}

We will use the above gradient descent method
in the computational examples below and assume that 
the initial guess $\phi_0$ is a small perturbation of $u(0)$.
Such an assumption can be relevant for many data assimilation applications. Indeed, it is typical that new observations need to be incorporated into the state of the system, and the current state can then be used as an initial guess. 

\subsection{The effect of regularization on the convergence in $\tau$}

\begin{figure}
	\centering
	\includegraphics[width=0.75\linewidth]{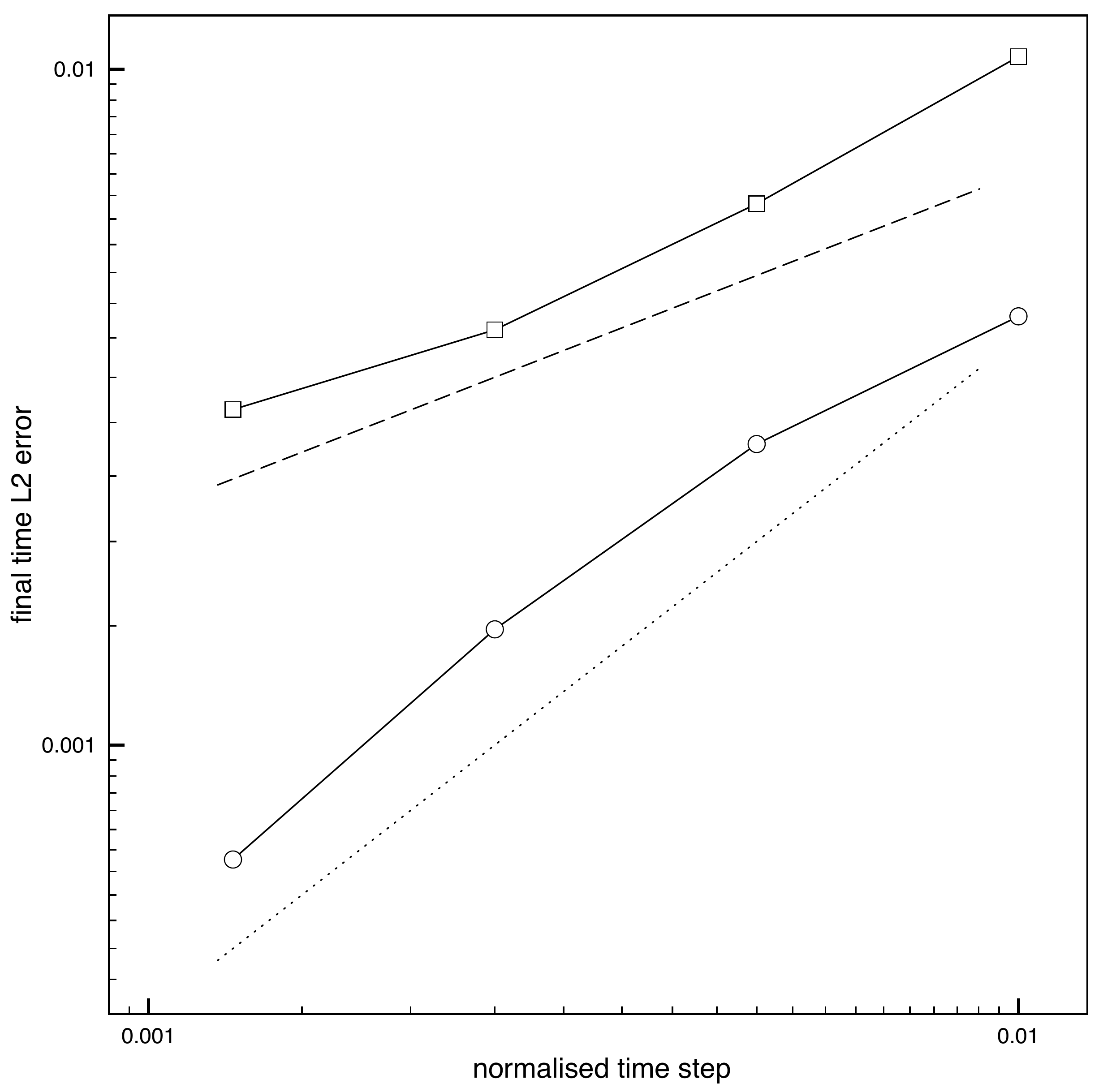}
\caption{ 
The effect of regularization on the convergence in $\tau$.
The convergence is of order $1/2$ (slope of dashed reference line) when
$\gamma_1 = 0$ (data with square markers) 
and of order $1$ (slope of dotted reference line) when $\gamma_1 = 1$ (data with circle markers) .
Here $\gamma_M = \gamma_0 = 1$, $h = 10^{-2}$,
and the error is $\norm{u(T) - u_h^N}_{L^2(\Omega)}$.
}
\label{fig:fullsolverconvergencefinal}
\end{figure}

\begin{figure}
	\centering
	\includegraphics[width=0.95\linewidth]{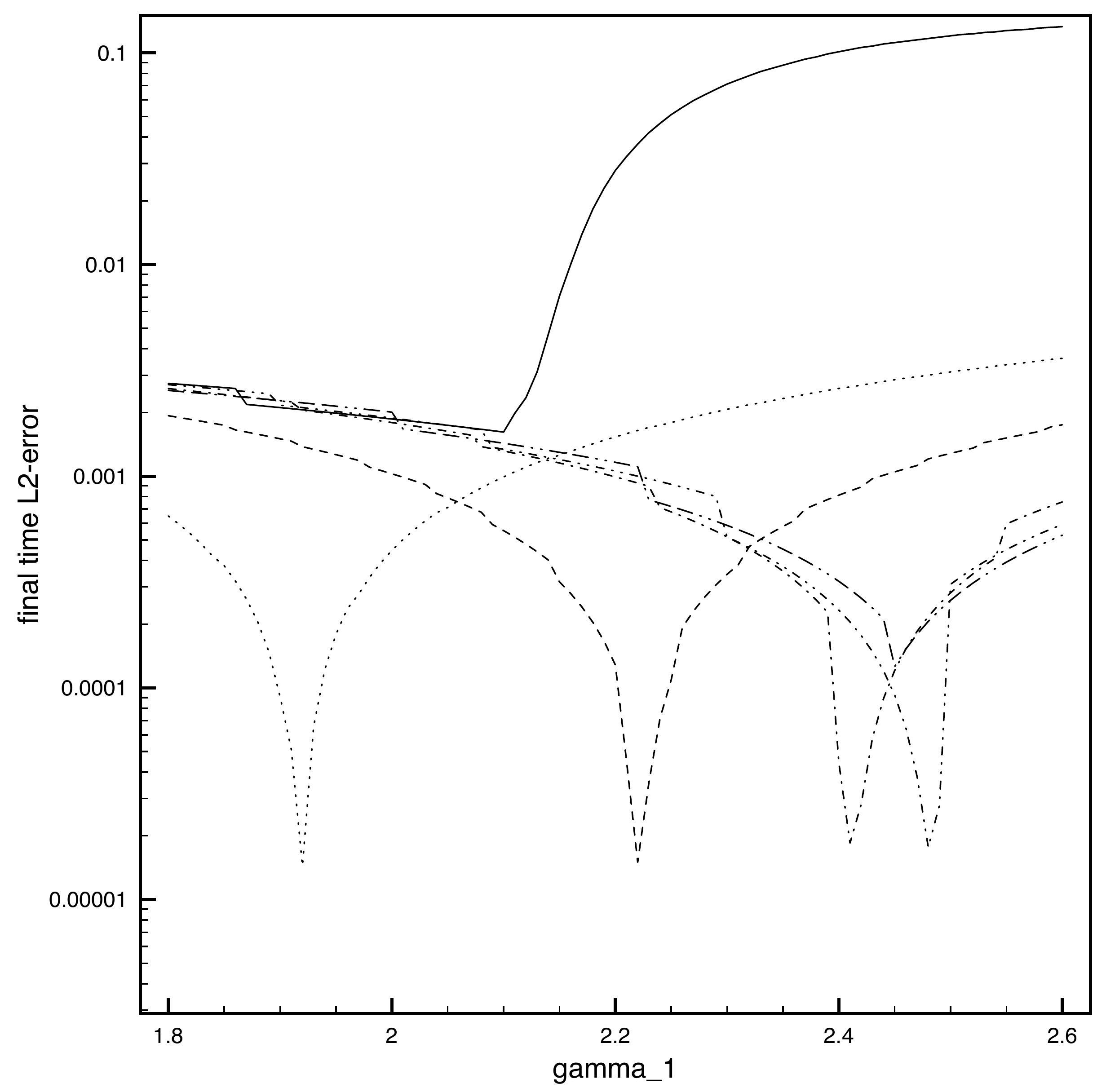}
\caption{The error for various choices of the constants $\gamma_0, \gamma_1$. Here $\gamma_M=1$, $h=\tau=10^{-2}$ and the error is $\norm{u(T) - u_h^N}_{L^2(\Omega)}$.
For each $0.1 \le \gamma_0 \le 1.2$, 
the method is robust for a large range
in $\gamma_1$. There also is an optimal value of $\gamma_1$ for each
such $\gamma_0$. However, this is mesh dependent and it is not clear
if the phenomenon can be exploited in practice. ($\gamma_0 = 0.1$ -
dotted line; $\gamma_0 = 0.2$ -
dashed line; $\gamma_0 = 0.6$ -
dash/dotted line; $\gamma_0 = 1.0$ -
dash/doubledotted line; $\gamma_0 = 1.2$ -
doubledash/doubledotted line; $\gamma_0 = 1.5$ -
filled line.)}
\label{fig:parameterstudyfinalgraph}
\end{figure}

We verified that the presence of the additional regularization in the case $\gamma_1 > 0$ leads to the improved convergence rate in $\tau$ as predicted by Theorem \ref{th_main}.
Indeed, in the computations summarized in Figure \ref{fig:fullsolverconvergencefinal}, 
the convergence is of order $1/2$ when $\gamma_1 = 0$
and of order $1$ when $\gamma_1 = 1$.
Here $\gamma_M = \gamma_0 = 1$, $h = 10^{-2}$,
$u$ is of the form (\ref{u_comp}) with $k=1$, and $T=0.1$.
We used the gradient descent method 
with the initial guess $\phi_0 = v + h$ where $v$ is the interpolation of $u(0)$ on $V_h$. The step size in (\ref{graddesc}) was taken $ \alpha = 0.1$ and the iteration (\ref{graddesc}) was terminated when $\norm{z^1}$ started to increase. 

\subsection{Sensitivity to the choice of $\gamma_0$ and $\gamma_1$.}
In all the numerical experiments above we have taken the parameters $\gamma_0$ and $\gamma_1$ to be either one or zero. This was to avoid special effects that can
appear due to parameter tuning. In a final numerical experiment we
verified that the method is not sensitive to the
particular choices of the constants $\gamma_0, \gamma_1 > 0$. The
conclusion of the study is that the method is robust for a wide range
of choices of $\gamma_0$ and $\gamma_1$, including
$\gamma_0=\gamma_1=1$. We observed that choosing both parameters large
resulted in solutions that were over regularized and yielded suboptimal
accuracy compared to lower values of the parameters. See the filled
line of Figure \ref{fig:parameterstudyfinalgraph} for an example.
We also observed that there are certain ``sweet spot'' combinations of values of $\gamma_0$ and $\gamma_1$
for which the errors are orders of magnitude smaller than for the
neighbouring parameter combinations. These optimal parameter
combinations however did not appear to be stable under mesh refinement
and it is unclear if this effect can be of any use in practice.
The computations are summarized in Figure
\ref{fig:parameterstudyfinalgraph}, with particular focus on the
parameter interval where the optimal parameter choices appeared.
Here $h = \tau = 10^{-2}$ and the other choices are as in the previous example. 



\bibliographystyle{abbrv}
\bibliography{main}


\ifoptionfinal{}{
\listoftodos
}
\end{document}